\documentclass[10pt, twoside]{amsart}

\usepackage{t1enc}
\usepackage[latin1]{inputenc}
\usepackage{latexsym}
\usepackage{amssymb}
\usepackage{graphicx}
\usepackage{amsmath}
\usepackage{amsthm}
\usepackage{amsfonts}
\usepackage{mathpazo}
\usepackage{mathrsfs}
\usepackage[paper = letterpaper, left = 3.5cm, right = 3.5cm, headsep = 6mm, footskip = 10mm,
top = 35mm, bottom = 35mm, footnotesep=5mm, headheight = 2cm]{geometry}
\usepackage{fancyhdr}
\usepackage[all]{xy}
\usepackage[british]{babel}
\usepackage{soul}
\usepackage{hyperref}

\newcommand*{\definiere}{\mathrel{\mathop:}=}

\newcommand*{\tensor}{\otimes}

\newcommand{\cyclic}{\mathop{\kern0.9ex{{+}\kern-2.10ex\raise-0.20
      ex\hbox{\Large\hbox{$\circlearrowright$}}}}\limits}

\def\N{\ifmmode{\mathbb N}\else{$\mathbb N$}\fi}
\def\Z{\ifmmode{\mathbb Z}\else{$\mathbb Z$}\fi}
\def\Q{\ifmmode{\mathbb Q}\else{$\mathbb Q$}\fi}
\def\R{\ifmmode{\mathbb R}\else{$\mathbb R$}\fi}
\def\C{\ifmmode{\mathbb C}\else{$\mathbb C$}\fi}
\def\K{\ifmmode{\mathbb K}\else{$\mathbb K$}\fi}
\def\P{\ifmmode{\mathbb P}\else{$\mathbb P$}\fi}
\def\g{\ifmmode{\mathfrak g}\else {$\mathfrak g$}\fi}
\def\h{\ifmmode{\mathfrak h}\else {$\mathfrak h$}\fi}
\def\a{\ifmmode{\mathfrak a}\else {$\mathfrak a$}\fi}
\def\k{\ifmmode{\mathfrak k}\else {$\mathfrak k$}\fi}
\def\p{\ifmmode{\mathfrak p}\else {$\mathfrak p$}\fi}
\def\b{\ifmmode{\mathfrak b}\else {$\mathfrak b$}\fi}
\def\n{\ifmmode{\mathfrak n}\else {$\mathfrak n$}\fi}
\def\m{\ifmmode{\mathfrak m}\else {$\mathfrak m$}\fi}
\def\t{\ifmmode{\mathfrak t}\else {$\mathfrak t$}\fi}
\def\O{\ifmmode{\mathscr{O}}\else {$\mathscr{O}$}\fi}
\def\W{\ifmmode{\mathcal{V}}\else {$\mathscr{W}$}\fi}

\def\hq{/\hspace{-0.14cm}/}
\def\kleinematrix#1,#2,#3,#4,{\begin{pmatrix}#1 & #2 \\ #3 & #4
  \end{pmatrix}}

\DeclareMathOperator{\tdim}{tdim}
\DeclareMathOperator{\supp}{supp}

\DeclareMathOperator{\Spec}{Spec}
\DeclareMathOperator{\Ass}{Ass}
\DeclareMathOperator{\Quot}{Quot}

\newtheoremstyle{daniel}{3.0mm}{0mm}{\itshape}{}{\bfseries}{.}{1.5mm}{}
\theoremstyle{daniel}
\newtheorem{thm}{Theorem}[section]
\newtheorem{prop}[thm]{Proposition}
\newtheorem{Defi}[thm]{Definition}
\newtheorem{lemma}[thm]{Lemma}

\newtheorem{Exs}[thm]{Examples}
\newtheorem{Ex}[thm]{Example}
\newtheorem{Rems}[thm]{Remarks}
\newtheorem{Rem}[thm]{Remark}

\newtheorem*{thm*}{Theorem}
\newtheorem*{cor*}{Corollary}
\newtheorem*{thm3.1}{Theorem 3.1}
\newtheorem*{prop*}{Proposition}

\newenvironment{rem}   {\begin{Rem}\em}{\end{Rem}}

\newenvironment{defi}  {\begin{Defi}\em}{\end{Defi}}
\newenvironment{ex}  {\begin{Ex}\em}{\end{Ex}}

\def\cprime{$'$} \def\polhk#1{\setbox0=\hbox{#1}{\ooalign{\hidewidth
  \lower1.5ex\hbox{`}\hidewidth\crcr\unhbox0}}}
  \def\polhk#1{\setbox0=\hbox{#1}{\ooalign{\hidewidth
  \lower1.5ex\hbox{`}\hidewidth\crcr\unhbox0}}}
\providecommand{\bysame}{\leavevmode\hbox to3em{\hrulefill}\thinspace}
\providecommand{\MR}{\relax\ifhmode\unskip\space\fi MR }
\providecommand{\MRhref}[2]{%
  \href{http://www.ams.org/mathscinet-getitem?mr=#1}{#2}
}
\providecommand{\href}[2]{#2}
\begin{document}
\title{1-rational singularities and quotients by reductive groups}
\author{Daniel Greb}
\date{\today}
\address{Mathematisches Institut\\
Abteilung f\"ur Reine Mathematik\\
Albert-Ludwigs-Universit\"at\\
Eckerstr.~1\\
79104 Freiburg im Breisgau\\
Germany}
\email{daniel.greb@math.uni-freiburg.de}{}
\urladdr{\href{http://home.mathematik.uni-freiburg.de/dgreb}{http://home.mathematik.uni-freiburg.de/dgreb}}

\begin{abstract}
We prove that good quotients of algebraic varieties with 1-rational singularities also have 1-rational singularities. This refines a result of Boutot on rational singularites of good quotients.
\end{abstract}
\maketitle
%\enlargethispage{5mm}
%
%
\section{Introduction and statement of the main result}
Generalising the Hochster-Roberts theorem \cite{HochsterRoberts}, Jean-Fran\c{c}ois Boutot proved that the class of varieties with rational singularities is stable under taking good quotients by reductive groups, see \cite{Boutot}.

In this short note we study varieties with \emph{1-rational singu\-la\-ri\-ties}. This is the natural class of singular varieties to which projectivity results for K\"ahler Moishezon manifolds generalise, cf.\ \cite{ProjectivityofMoishezon}. Our main result is:
\begin{thm3.1}
Let $G$ be a complex reductive Lie group and let $X$ be an algebraic $G$-variety such that the
good quotient $\pi: X \to X\hq G$ exists. If $X$ has $1$-rational
singularities, then $X\hq G$ has $1$-rational singularities.
\end{thm3.1}
In our proof, which forms a part of the author's thesis \cite{Doktorarbeit}, we follow \cite{Boutot} and we check that in Boutot's arguments it is possible to separate the different cohomology degrees.

Theorem~\ref{rationalsingularities} has been used in \cite{PaHq} to prove projectivity of compact momentum map quotients of algebraic varieties.

\subsection*{Acknowledgements}
The author wants to thank Miles Reid for kindly answering his questions via e-mail. The author gratefully acknowledges the financial support of the Mathematical Sciences Research Institute, Berkeley, by means of a postdoctoral fellowship during the 2009 "Algebraic Geometry" program.
\section{Preliminaries}
\subsection{Singularities and resolution of singularities}
We work over the field $\C$ of complex numbers.  If $X$ is an algebraic variety we denote by $\tdim_x X$ the dimension of the
Zariski tangent space at $x \in X$. I.e., if $\mathfrak{m}_x$ denotes the maximal ideal in
$\mathscr{O}_{X,x}$, then $\tdim_x X \definiere \dim_\C
\mathfrak{m}_x / \mathfrak{m}_x^2$. A point $x$ in an algebraic variety $X$ is
called \emph{singular} if $\tdim_xX > \dim X$. We denote the singular set of an algebraic variety $X$ by $X_{sing}$. A variety $X$ is called \emph{smooth} if $X_{sing} = \emptyset$.
\begin{defi}
Let $X$ be an algebraic variety. A \emph{resolution} of $X$ is a proper birational
surjective morphism $f: Y
\rightarrow X$ from a smooth algebraic variety $Y$ to $X$.
\end{defi}
If $X$ is an algebraic variety, by a theorem of Hironaka \cite{Hironaka} there exists a resolution $f: Y \to X$ by a projective morphism $f$ such that the
restriction $f: f^{-1}(X\setminus X_{sing}) \to X \setminus
X_{sing}$ is an isomorphism. See also \cite{Bierstone}, \cite{Hauser}, and \cite{KollarResolutions} for later improvements and simplifications of Hironaka's proof.
\subsection{$1$-rational singularities}\label{1rationalsing}
In this section we introduce the class of singularities studied in this note.
\begin{defi}
An algebraic variety $X$ is said to have \emph{$1$-rational
singularities}, if the following two conditions are fulfilled:
\vspace{-3mm}
\begin{enumerate}
\item $X$ is normal,
\item for every resolution $f: \widetilde X \rightarrow X$ of $X$ we have $R^1f_*\mathscr O _{\widetilde X} = 0$.
\end{enumerate}
\end{defi}
\begin{prop}\label{rationalresolutionrationalsingularities}
Let $X$ be a normal algebraic variety. If there exists one resolution
$f_0: X_0
\rightarrow X$ of $X$ such that $R^1(f_0)_*\mathscr{O}_{X_0} = 0$, then $X$ has $1$-rational singularities.
\end{prop}
\begin{proof}
Let $f_1: X_1 \rightarrow X$ be a second resolution of $X$. By \cite{Hironaka}, there exits a smooth
algebraic variety $Z$ and resolutions $g_0: Z \rightarrow X_0$ and $g_1: Z \rightarrow X_1$ such
that the following diagram commutes
\[\begin{xymatrix}{
   &  Z  \ar[rd]^{g_1}\ar[ld]_{g_0}&    \\
X_0\ar[rd]_{f_0}&     & X_1\ar[ld]^{f_1} \\
   &  X  & .
}
  \end{xymatrix}
\]
For $j=0,1$ there exists a spectral sequence (see \cite{Weibel}) with lower terms
\[0 \rightarrow R^1 (f_j)_*((g_j)_*\mathscr{O}_Z) \rightarrow R^1(f_j\circ g_j)_*\mathscr{O}_Z
\rightarrow (f_j)_*(R^1(g_j)_*\mathscr{O}_Z) \rightarrow \cdots. \]
Since $g_0$ and $g_1$ are resolutions of smooth algebraic varieties, we have
$R^1(g_j)_*\mathscr{O}_Z =0$ and $(g_j)_*\mathscr{O}_Z = \mathscr{O}_{X_j}$ for $j =0,1$ (see
\cite{Hironaka} and \cite{Ueno}). It follows that
\[0= R^1(f_0)_*\mathscr{O}_{X_0}\cong R^1 (f_0\circ g_0)_*\mathscr{O}_Z \cong R^1(f_1 \circ
g_1)_*\mathscr{O}_Z \cong R^1(f_1)_*\mathscr{O}_{X_1}.\qedhere\]
\end{proof}
\begin{rem}
The proof shows that, if $f_1: X_1 \rightarrow X$ and $f_2:
X_2 \rightarrow X$ are two resolutions of $X$, there is an
isomorphism $R^1(f_1)_* \mathscr{O}_{X_1} \cong R^1(f_2)_*
\mathscr{O}_{X_2}$. From this, it follows that having $1$-rational singularities is a local property.
Furthermore, if $X$ is an algebraic $G$-variety for an algebraic group $G$, and $f: \widetilde {X} \to
X$ is a resolution of $X$, then the support of $R^1f_*\mathscr{O}_X$
is a $G$-invariant subvariety of $X$.
\end{rem}
\subsubsection{Rational singularities}
In this section we shortly discuss the relation of the notion "1-rational singularity" to the more commonly used notion of "rational singularity".
\begin{defi}
An algebraic variety $X$ is said to have \emph{rational
singularities}, if the following two conditions are fulfilled:
\vspace{-3mm}
\begin{enumerate}
\item $X$ is normal,
\item for every resolution $f: \widetilde X \rightarrow X$ of $X$ we have $R^jf_*\mathscr O _{\widetilde X} = 0
$ for all $j =1, \dots, \dim X$.
\end{enumerate}
\end{defi}
\begin{rem}
Again, the vanishing of the higher direct image sheaves is independent of the chosen resolution.
\end{rem}
Due to a result of Malgrange \cite[p.\ 236]{Malgrange} asserting that $R^{\dim  X}f_*\mathcal{O}_{\widetilde{X}} = 0$ for every resolution $f: \widetilde X \to X$ of an irreducible variety $X$, an algebraic surface has $1$-rational singularities if and only if it has rational singularities. These notions differ in higher dimensions as is illustrated by the following example.
\begin{ex}
Let $Z$ be a smooth quartic hypersurface in $\P_3$
and let $X$ be the affine cone over $Z$ in $\C^4$. The variety $X$ has an isolated singularity at the origin. Let $L$ be the total space of the line bundle $\mathcal{O}_Z(-1)$, i.e., the restriction of the dual of the hyperplane bundle of $\P_3$ to $Z$. Then, blowing down the zero section $\mathcal{Z}_L \subset L$ and setting $\widetilde X \definiere L$, we obtain a map
$f: \widetilde X \to X$ which is a resolution of singularities, an isomorphism outside of the origin $0 \in X$ with $f^{-1}(0) = \mathcal{Z}_L \cong Z$. We claim that $0 \in X$ is a $1$-rational singularity which is not rational.

To see that the origin is a normal point of $X$ it suffices to note that it is obtained as the blow-down of the maximal compact subvariety $\mathcal{Z}_L$ of the smooth variety $L$.

To compute $(R^jf_*\mathcal{O}_{\widetilde X})_0$, we use the Leray spectral sequence
\[\cdots \to H^j(X, \mathcal{O}_X) \to H^j(\widetilde{X}, \mathcal{O}_{\widetilde X}) \to H^0(X, R^j f_* \mathcal{O}_{\widetilde X}) \to H^{j+1}(X, \mathcal{O}_X)    \to  \cdots\]
and the fact that $X$ is affine to show that $H^j(\widetilde{X}, \mathcal{O}_{\widetilde X}) \cong H^0(X, R^jf_*\mathcal{O}_{\widetilde X}) = (R^jf_*\mathcal{O}_{\widetilde X})_0$. Expanding cohomology classes into Taylor series along fibres of $L$, we get that $H^j(\widetilde{X}, \mathcal{O}_{\widetilde X}) \cong \bigoplus_{k\geq 0}H^j(Z, \mathcal{O}_Z(k))$. Hence, we have
\begin{equation}\label{hypersurfacecohomology}
(R^jf_*\mathcal{O}_{\widetilde X})_0 \cong \bigoplus_{k\geq 0}H^j(Z, \mathcal{O}_Z(k)) \quad \quad\text{for all $j \geq 1$}.
\end{equation}
It follows from \eqref{hypersurfacecohomology} and \cite[Chap III, Ex 5.5]{Hartshorne} that $(R^1f_*\mathcal{O}_{\widetilde X})_0 = 0$, and hence that $0 \in X$ is a $1$-rational singularity. Since the canonical bundle $\mathcal{K}_Z$ of $Z$ is trivial, it follows from Serre duality that $H^2(Z, \mathcal{O}_Z(k)) \cong H^0(Z, \mathcal{O}_Z(-k))$. As a consequence, we get
\begin{equation}
H^2(Z, \mathcal{O}_Z(k)) =
\begin{cases}
\;\mathbb{C} & \text{for } k=0,\\
\;0  & \text{otherwise}.
\end{cases}
\end{equation}
Together with \eqref{hypersurfacecohomology} this implies that $(R^2f_*\mathcal{O}_{\widetilde X})_0 = \C$. Consequently, the singular point $0 \in X$ is not a rational singu\-la\-ri\-ty.
\end{ex}
\subsection{Good quotients}
\begin{defi}
Let $G$ be a complex reductive Lie group acting algebraically on an algebraic variety $X$. An algebraic variety $Y$ together with a morphism $\pi: X \to Y$ is called \emph{good quotient} of $X$ by the action of $G$, if
\begin{enumerate}
\item $\pi$ is $G$-invariant, surjective, and affine,
\item $(\pi_* \mathscr{O}_X)^G=\mathscr{O}_Y$.
\end{enumerate}
\end{defi}
\begin{ex}
Let $X$ be an affine $G$-variety. Then the algebra of invariants $\C [X]^G$ is finitely generated, and its inclusion into $\C[X]$ induces a regular map $\pi$ from $X$ to $Y \definiere \Spec (\C[X]^G)$ which fulfills $(1)$ and $(2)$ above, i.e., $\pi$ is a good quotient.
\end{ex}
\section{Singularities of good quotients: proof of the main result}
Let $G$ be a complex reductive Lie group and let $X$ be an algebraic $G$-variety such that the
good quotient $X\hq G$ exists. We study the singularities of $X\hq G$ relative
to the singularities of $X$.

More precisely, we prove the following theorem, which is the main result of this note.
\begin{thm}\label{rationalsingularities}
Let $G$ be a complex reductive Lie group and let $X$ be an algebraic $G$-variety such that the
good quotient $\pi: X \to X\hq G$ exists. If $X$ has $1$-rational
singularities, then $X\hq G$ has $1$-rational singularities.
\end{thm}
Before proving the theorem we explain two technical lemmata. The first one discusses the relation between cohomology modules of $X$ and of $X\hq G$.
\begin{lemma}\label{integratingcohomology}
Let $X$ be an algebraic $G$-variety with good quotient $\pi: X \to X\hq G$. Then the
natural map $\pi^*: H^1(X\hq G, \mathscr{O}_{X /\negthickspace / G}) \to H^1(X, \mathscr{O}_X)$ is injective.
\end{lemma}
\begin{proof}Let $\mathscr{U} = \{U_i\}_{i\in I}$ be an affine open covering of $X\hq G$. We
can compute the cohomology module $H^1(X\hq G, \mathscr{O}_{X /\negthickspace/ G})$
via \v{C}ech cohomology with
respect to the covering $\mathscr{U}$. Since $\pi$ is an affine map, $\pi^{-1}(\mathscr{U})
\definiere \{\pi^{-1}(U_i)\}_{i \in I}$ is an affine open covering of $X$ and we can compute the
cohomology module $H^1(X, \mathscr{O}_X)$ via \v{C}ech cohomology with respect to the covering
$\pi^{-1}(\mathscr{U})$.

Let $\eta = (\eta_{ij}) \in C^1(\mathscr{U}, \mathscr{O}_{X /\negthickspace / G})$ be a \v{C}ech cocycle such that
the pullback of the associated cohomology class $[\eta] \in H^1(X\hq G, \mathscr{O}_{X /\negthickspace / G})$
fulfills $\pi^* ([\eta]) = 0 \in H^1(X, \mathscr{O}_X)$. Then, there exists a cocycle $\nu =
(\nu_i) \in C^0(\pi^{-1}(\mathscr{U}), \mathscr{O}_X)$ such that
\[\pi^* (\eta_{ij}) = \nu_i|_{\pi^{-1}(U_{ij})} - \nu_j|_{\pi^{-1}(U_{ij})}  \in
\mathscr{O}_X(\pi^{-1}(U_{ij})).\] Averaging $\nu_i \in \mathscr{O}_X(\pi^{-1}(U_i))$ over a
maximal compact subgroup $K$ of $G$ we obtain invariant functions $\widetilde \nu_i \in
\mathscr{O}_X(\pi^{-1}(U_i))^G$. Since $\pi^*(\eta_{ij}) \in \mathscr{O}_X(\pi^{-1}(U_{ij}))^G$, the cocycle $\widetilde{\nu} = (\widetilde{\nu}_i) \in C^0(\pi^{-1}(\mathscr{U}), \mathscr{O}_X)$ fulfills
\[\pi^*(\eta_{ij}) = \widetilde \nu_i|_{\pi^{-1}(U_{ij})} - \widetilde \nu_j|_{\pi^{-1}(U_{ij})}  \in
\mathscr{O}_X(\pi^{-1}(U_{ij})).\]
For all $i$, there exist a uniquely determined function $\widehat \nu_i \in \mathscr{O}_{X /\negthickspace / G}(U_i)$
with $\pi^*(\widehat \nu_i) = \widetilde{\nu}_i$. Consequently, we have
\[\eta_{ij} = \widehat \nu_i|_{U_{ij}} - \widehat \nu_j|_{U_{ij}} \in \mathscr{O}_{X /\negthickspace / G}(U_{ij}).\]
Therefore, $[\eta] = 0 \in H^1(X\hq G, \mathscr{O}_{X /\negthickspace / G})$ and $\pi^*$ is injective, as claimed.
\end{proof}

The second lemma will be used to obtain information about the singularities of an algebraic variety $X$ from information about the singularities of a general hyperplane section $H$ of $X$ and vice versa.

\begin{lemma}\label{projectionformula}
Let $X$ be a normal affine variety with $\dim X \geq 2$ and let $f: \widetilde{X} \to X$ be a resolution of singularities. Let $\mathscr{L} \subset \mathscr{O}_X(X)$ be a finite-dimensional subspace, such that the associated linear system is base-point free, and such that the image of $X$ under the associated map $\varphi_{\mathscr{L}}: X \to \P_N$ fulfills $\dim\bigl(\varphi_{\mathscr{L}}(X) \bigr) \geq 2$. If $h \in \mathscr{L}$ is a general element, then the following holds for the corresponding hyperplane section $H \subset X$:
\vspace{-2.5mm}
\begin{enumerate}
\item The preimage $\widetilde {H} \definiere f^{-1}(H)$ is smooth and $f|_{\widetilde H}: \widetilde{H} \to H$ is a resolution of $H$.
\item We have $R^jf_*\mathscr{O}_{\widetilde{H}} \cong R^jf_*\mathscr{O}_{\widetilde X} \tensor \mathscr{O}_{\overset{}{H}}\,$ for $j = 0,1$.
\end{enumerate}
\end{lemma}
\begin{proof}
1) This follows from Bertini's theorem (see \cite[Chap III, Cor 10.9]{Hartshorne}).

2) In the exact sequence
\begin{equation}\label{idealsequence}
0 \to \mathscr{O}_{\widetilde X}(-\widetilde H) \overset{m}{\to} \mathscr{O}_{\widetilde X} \to \mathscr{O}_{\widetilde H} \to 0,
\end{equation}
the map $m$ is given by multiplication with the equation $h\in \mathcal{O}_X(X)$ defining $H$ and $\widetilde{H}$. Pushing forward the short exact sequence \eqref{idealsequence} by $f_*$ yields the long exact seqence
\begin{align}\label{longsheafsequence}
\begin{split}
0 &\to f_* \mathscr{O}_{\widetilde X}(-\widetilde H) \overset{m_0}{\to} f_*\mathscr{O}_{\widetilde X} \to f_*\mathscr{O}_{\widetilde H} \to \\
  &\to R^1f_* \mathscr{O}_{\widetilde X}(-\widetilde H) \overset{m_1}{\to} R^1 f_*\mathscr{O}_{\widetilde X} \to R^1 f_*\mathscr{O}_{\widetilde H} \to \\
  &\to R^2f_* \mathscr{O}_{\widetilde X}(-\widetilde H) \overset{m_2}{\to} R^2 f_*\mathscr{O}_{\widetilde X} \to R^2 f_*\mathscr{O}_{\widetilde H} \to \\
  &\to \cdots.
\end{split}
\end{align}
Since $X$ is an affine variety, the exact sequence above is completely determined by the following sequence of finite $\mathscr{O}_X(X)$-modules:
\begin{align}\label{longsequence}
\begin{split}
0 &\to \Gamma(\widetilde X, \mathscr{O}_{\widetilde X}(-\widetilde H)) \overset{m_0}{\to} \Gamma(\widetilde X, \mathscr{O}_{\widetilde X}) \to \Gamma(\widetilde X,\mathscr{O}_{\widetilde H}) \to \\
 &\to H^1(\widetilde X, \mathscr{O}_{\widetilde X}(-\widetilde H) ) \overset{m_1}{\to} H^1(\widetilde X,\mathscr{O}_{\widetilde X} )  \to H^1(\widetilde X,\mathscr{O}_{\widetilde H} )  \to \\
 &\to H^2(\widetilde X,\mathscr{O}_{\widetilde X}(-\widetilde H) ) \overset{m_2}{\to}H^2(\widetilde X,\mathscr{O}_{\widetilde X} ) \to H^2(\widetilde X, \mathscr{O}_{\widetilde H}) \to \\
 & \to \cdots.
\end{split}
\end{align}
The maps $m_j$, $j=0,1,2$ are given by multiplication with the element $h \in \mathscr{O}_X(X)$. We claim that we can choose $h \in \mathscr{L}$ in such a way that $m_1$ and $m_2$ are injective. Indeed, if $R$ is a commutative Noetherian ring with unity, $M$ is a finite $R$-module, and $Z_R(M)$ denotes the set of zerodivisors for $M$ in $R$, we have
\[Z_R(M) =  \bigcup_{P \in \Ass M} P,\]
where $\Ass M$ is the finite set of assassins (or associated primes) of $M$ . Hence, the set of zerodivisors for $H^1(\widetilde X, \mathscr{O}_{\widetilde X}(-\widetilde H) )$ and $H^2(\widetilde X,\mathscr{O}_{\widetilde X}(-\widetilde H) )$ is a union of finitely many prime ideals $P$ of $\mathscr{O}_X(X)$. Since the linear system associated to $\mathscr{L}$ is base-point free, the general element $h$ of $\mathscr{L}$ lies in $\mathscr{L}\setminus \bigcup P$. For such a general $h \in \mathscr{L} \setminus \bigcup P$, the maps $m_1$ and $m_2$ in the sequences \eqref{longsequence} and \eqref{longsheafsequence} are injective.

Since $\mathscr{O}_{\widetilde X}(- \widetilde{H}) \cong f^* (\mathscr{O}^{}_X(-H))$, the projection formula for locally free sheaves (see \cite[Chap III, Ex 8.3]{Hartshorne}) yields
\begin{align*}
R^jf_*\mathscr{O}_{\widetilde X}(- \widetilde H) \cong R^jf_*\mathscr{O}_{\widetilde X} \tensor \mathscr{O}_{X}^{}(-H).
\end{align*}
Furthermore, for $j=0,1$ the image of $m_j$ coincides with the image $\mathscr{B}_j$ of the natural map $R^jf_*\mathscr{O}_{\widetilde X} \otimes \mathscr{O}_{X}^{}(-H) \to R^jf_*\mathscr{O}_{\widetilde X}$.  Since $m_1$ and $m_2$ are injective by the choice of $H$, it follows that
\[R^jf_*\mathscr{O}_{\widetilde H} \cong R^jf_*\mathscr{O}_{\widetilde X} / \mathscr{B}_j \quad \text{for } j=0,1.\]
Tensoring with $R^jf_*\mathscr{O}_{\widetilde X}$ is right-exact, and hence the exact sequence
\[0 \to \mathscr{O}_X(-H) \to \mathscr{O}_X \to \mathscr{O}_H \to 0\]
yields $R^j f_* \mathscr{O}_{\widetilde H} \cong R^1f_*\mathscr{O}_{\widetilde X} \tensor \mathscr{O}^{}_H$, as claimed.
\end{proof}
\begin{proof}[Proof of Theorem \ref{rationalsingularities}]
Since the claim is local and $\pi$ is an affine map, we may assume that $X\hq G$ and $X$ are
affine.

First, we prove that normality of $X$ implies normality of $X\hq G$. We have to show that $\C[X\hq
G]\cong \C[X
]^G$ is a normal ring. So let $\alpha \in \Quot(\C[X]^G) \subset \C(X)^G$ be an element of the
quotient field of $\C[X]^G$ and assume that $\alpha$ fulfills a monic equation
\[\alpha^n + c_1\alpha^{n-1} + \dots + c_n = 0 \]with coefficients
$c_j \in \C[X]^G \subset \C[X]$. Since $\C[X]$
is normal by assumption, it follows that $\alpha \in \C(X)^G \cap \C[X] = \C[X]^G$. Hence,
$\C[X]^G$ is normal. As a consequence, we can assume in the following that $X$ is $G$-irreducible.

We prove the claim by induction on $\dim X\hq G$. For $\dim X\hq G = 0$ there is nothing to show.
For $\dim X\hq G = 1$ we notice that $X\hq G$ is smooth. So, let $\dim X\hq G \geq 2$. Let
$\pi: X \to X\hq G$ denote the quotient map and let $p_X: \widetilde X \to X$ be a resolution of $X$.
First, we prove that a general hyperplane section $H \subset X\hq G$ has $1$-rational
singularities. If $H$ is a general hyperplane section in $X\hq G$, Lemma \ref{projectionformula} applied to $\pi^{-1}(H)$ yields that $p_X|_{\widetilde H}: \widetilde H \to \pi^{-1}(H)$ is a resolution, where $\widetilde H = p_X^{-1}(\pi^{-1}(H))$, and that
\[R^j(p_X)_*\mathscr{O}_{\widetilde{X}} \otimes_{\mathscr{O}_X}\mathscr{O}_{\pi^{-1}(H)} =
R^j (p_X)_*\mathscr{O}_{\widetilde{H}} \quad \text{ for } j=0,1. \] Since $f_* \mathscr{O}_{\widetilde{X}}=\mathscr{O}_X$ by
Zariski's main theorem, it follows from the case $j=0$ that $\pi^{-1}(H)$ is normal. Alternatively, one could invoke Seidenberg's Theorem (see e.g.\ \cite[Thm.\ 1.7.1]{SommeseAdjunction}). Together with the case $j=1$, this implies that
$\pi^{-1}(H)$ has $1$-rational singularities. By
induction, it follows that $H = \pi^{-1}(H)\hq G$ has $1$-rational singularities.

Let $p: Z \to X\hq G$ be a resolution of $X\hq G$. As we have seen above, a general hyperplane
section $H$ of $X\hq G$ has $1$-rational singularities and the restriction of $p$ to $\widehat{H} \definiere
p^{-1 }(H)$, $p|_{p^{-1}(H)}: \widehat H \to H$ is a resolution of $H$. It follows that $\mathscr{O}_H
\otimes R^1p_*\mathscr{O}_Z = R^1p_*\mathscr{O}_{\widehat{H}}=0$. Consequently, the support of
$R^1p_*\mathscr{O}_Z$ does not intersect $H$ and hence, $\supp(R^1p_*\mathscr{O}_Z)$ consists of isolated points. Since the
claim is local, we can assume in the following that $R^1p_*\mathscr{O}_Z$ is supported at a single
point $x_0 \in X\hq G$.

The group $G$ acts on the fibre product $Z\times_{X /\negthickspace / G} X$ such that the map $p_X: Z \times_{X /\negthickspace/ G}
X \to X$ is equivariant. One of the $G$-irreducible components $\widetilde{X}$ of $Z\times_{X /\negthickspace/ G} X$
is birational to $X$, and, by passing to a resolution of $\widetilde{X}$ if necessary, we can assume that
$p_X: \widetilde{X} \to X$ is a resolution of $X$. We obtain the following commutative diagram
\[\begin{xymatrix}{
X \ar[d]_\pi & \ar[l]_{p_X} \ar[d]^{p_Z}\widetilde{X} \\
X\hq G & \ar[l]_<<<<<{p} Z .
}
  \end{xymatrix}
\]
Since $R^1p_*\mathscr{O}_Z$ is supported only at $x_0$, we have $(R^1p_*\mathscr{O}_Z)_{x_0} = H^0(X\hq G, R^1p_*\mathscr{O}_Z)$. Recall that $X\hq G$ is affine, hence, the Leray spectral sequence
\[0 \to H^1(X\hq G, \mathscr{O}_{X /\negthickspace / G}) \to H^1(Z, \mathscr{O}_Z) \to H^0(X\hq G,R^1p_*\mathscr{O}_Z ) \to H^2(X\hq G, \mathscr{O}_{X /\negthickspace / G})\to \cdots\]
implies that it suffices to show that $H^1(Z, \mathscr{O}_Z) = 0$.

Since $X$ is affine and has $1$-rational singularities, it follows
from the Leray spectral sequence
\[0 \to H^1(X,\mathscr{O}_X) \to H^1(\widetilde{X},
\mathscr{O}_{\widetilde{X}}) \to H^0(X, R^1f_*\mathscr{O}_{\widetilde{X}})\to \cdots\]
that $H^1(\widetilde{X}, \mathscr{O}_{\widetilde{X}}) \cong H^1(X, \mathscr{O}_X) = 0$. Consequently, it
suffices to show that there exists an injective map $H^1(Z, \mathscr{O}_Z) \hookrightarrow H^1
(\widetilde{X}, \mathscr{O}_{\widetilde{X}})$.

We introduce the following notation: $ U= (X\hq G) \setminus \{x_0\}$, $U' =\pi^{-1}(U) \subset X$, $
\widetilde{U}= p_X^{-1}(U') \subset \widetilde{X}$, $V = p^{-1}(U) \subset Z$.
We obtain the following commutative diagram of canonical maps
\begin{equation}\label{diagramrational}
\begin{xymatrix}{
 H^1(\widetilde{X}, \mathscr{O}_{\widetilde{X}}) \ar[r]^{h_{\widetilde X , \widetilde U}}& H^1 (\widetilde{U},
\mathscr{O}_{\widetilde{U}}) & \ar[l]_{h_{U',\widetilde U}} H^1(U', \mathscr{O}_{U'}) \\
H^1(Z, \mathscr{O}_Z) \ar[u]_{h_{Z,\widetilde{X}}}\ar[r]^{h_{Z,V}}& H^1(V, \mathscr{O}_V) \ar[u]&
H^1(U, \mathscr{O}_U)\ar[l]_{h_{U,V}}\ar[u]_{h_{U,U'}}.
}
  \end{xymatrix}
\end{equation}
Let $Y \definiere p^{-1}(x_0) \subset Z$. If $H^1_Y(Z, \mathscr{O}_Z)$
denotes the local cohomology groups with support in $Y$, there exists an exact sequence
\begin{equation}\label{localcohomologysequence}\dots \to H^1_Y(Z,\mathscr{O}_Z) \to H^1 (Z,
\mathscr{O}_Z) \overset{h_{Z,V}}{\longrightarrow} H^1(V, \mathscr{O}_V) \to \cdots,
\end{equation}
see \cite[Chap III, Ex 2.3]{Hartshorne}.
We use the following vanishing result which is originally due to Hartshorne and Ogus \cite{HartshorneOgus}.
\begin{prop}\label{localvanishing}
Let $W$ be a normal affine algebraic variety and $p: Z \to W$ a resolution. Let $w \in W$ and $Y = p^{-1}(w)$. Then $H^i_Y(Z, \mathscr{O}_Z) = 0$ for all $i < \dim W$.
\end{prop}
\begin{proof}[Sketch of proof]
Set $n \definiere \dim W$. As a first step, we compactifiy $p$. There exist projective completions $\overline{Z}$ and $\overline{W}$ of $Z$ and $W$, respectively, and a resolution $\overline{p}: \overline{Z} \to \overline{W}$ such that $\overline{p}^{-1}(W) = Z$ and $\overline{p}|_{Z} = p$. Since $Z$ is an open neighourhood of $Y$ in $\overline{Z}$, the Excision Theorem of Local Cohomology (see \cite[Chap III, Ex 2.3]{Hartshorne}) implies that $H_Y^j(\overline{Z},\mathscr{O}_{\overline{Z}}) = H_Y^j(Z,\mathscr{O}_Z)$. Let $\mathscr{K}_{\overline{Z}}$ be the locally free sheaf associated to the canonical bundle of $\overline{Z}$. The Formal Duality Theorem (see \cite{HartshorneAmple}) implies that the dual $H_Y^j(\overline{Z}, \mathscr{O}_{\overline{Z}})^*$ of $H_Y^j(\overline{Z}, \mathscr{O}_{\overline{Z}})$ is isomorphic to $(R^{n-j}\overline{p}_*\mathscr{K}_{\overline{Z}})_w^{\widehat\ }$, where $\ ^{\widehat\ }$ denotes completion with respect to the maximal ideal of $\mathscr{O}_{\overline{W}, w} = \mathscr{O}_{W,w}$. In summary, we have obtained an isomorphism
\[H_Y^j(Z, \mathscr{O}_Z)^* \cong (R^{n-j}\overline{p}_*\mathscr{K}_{\overline{Z}})_w^{\widehat\  }\quad \quad \text{for all } j = 0, \dots, n. \]
By Grauert-Riemenschneider vanishing (see e.g.\ \cite[Chap 4.3.B]{LazarsfeldPositivity1}), the term on the right hand side equals zero for $n-j \geq 1$. This proves the claim.
\end{proof}
Since $\dim X\hq G \geq 2$, Proposition \ref{localvanishing} yields $H^1_Y(Z, \mathscr{O}_Z) = 0$. As a consequence of
\eqref{localcohomologysequence}, the map $h_{Z,V}$ is injective.

The restriction of $p$ to $V = p^{-1}(U)$ is a resolution of $U$. Since the support of
$R^1p_*\mathscr{O}_Z$ is concentrated at $x_0$, the variety $U$ has $1$-rational singularities, and the Leray
spectral sequence
\[0 \to H^1(U, \mathscr{O}_U) \overset{h_{U,V}}{\longrightarrow} H^1(V, \mathscr{O}_V) \to H^0(U,
R^1p_*\mathscr{O}_U) \to \cdots \]
yields that $h_{U,V}$ is bijective. Similar arguments show that $h_{U',\widetilde U}$ is
bijective. Furthermore, Lemma \ref{integratingcohomology} implies that $h_{U, U'}$ is injective.

By the considerations above, the map
\[h_{Z,\widetilde{U}} \definiere h_{U',\widetilde{U}}^{} \circ h_{U,U'}^{} \circ h^{-1}_{U,V} \circ h_{Z,V}^{}\]
is injective. Diagram \eqref{diagramrational} implies $h_{Z, \widetilde{U}} = h_{\widetilde{X},
\widetilde{U}} \circ h_{Z, \widetilde{X}}$, and therefore $h_{Z,\widetilde{X}}$ is injective. Consequently, we
have $H^1(Z, \mathscr{O}_Z) = 0$. This concludes the proof of Theorem \ref{rationalsingularities}.
\end{proof}
\def\cprime{$'$} \def\polhk#1{\setbox0=\hbox{#1}{\ooalign{\hidewidth
  \lower1.5ex\hbox{`}\hidewidth\crcr\unhbox0}}}
  \def\polhk#1{\setbox0=\hbox{#1}{\ooalign{\hidewidth
  \lower1.5ex\hbox{`}\hidewidth\crcr\unhbox0}}}
\providecommand{\bysame}{\leavevmode\hbox to3em{\hrulefill}\thinspace}
\providecommand{\MR}{\relax\ifhmode\unskip\space\fi MR }
\providecommand{\MRhref}[2]{%
  \href{http://www.ams.org/mathscinet-getitem?mr=#1}{#2}
}
\providecommand{\href}[2]{#2}

\end{document}